\documentclass[11pt]{amsart}
\usepackage{cite,hyperref}
\usepackage[text={6.1in,8.5in},centering]{geometry}
\usepackage{amssymb,amsmath,amsthm,mathtools,extpfeil}
\usepackage[all,cmtip]{xy}
\usepackage{tikz}
\usepackage{rotating,enumerate,subcaption}

\usetikzlibrary{arrows}

\newtheorem{thm}[equation]{Theorem}
\newtheorem*{thm*}{Theorem}
\newtheorem{thmA}{Theorem}

\newtheorem{lem}[equation]{Lemma}

\newtheorem*{prop*}{Proposition}

\newtheorem{conj}[equation]{Conjecture}
\theoremstyle{definition}

\newtheorem{ex}[equation]{Example}

\numberwithin{equation}{section}

\DeclareMathOperator{\id}{id}

\DeclareMathOperator{\vol}{vol}

\DeclareMathOperator{\intr}{int}

\DeclareMathOperator{\rk}{rk}

\DeclareMathOperator{\Lip}{Lip}

\newcommand{\epsi}{\varepsilon}

\begin{document}
\title{A zoo of growth functions of mapping class sets}
\author{Fedor Manin}
\begin{abstract}
  Suppose $X$ and $Y$ are finite complexes, with $Y$ simply connected.  Gromov
  conjectured that the number of mapping classes in $[X,Y]$ which can be realized
  by $L$-Lipschitz maps grows asymptotically as $L^\alpha$, where $\alpha$ is an
  integer determined by the rational homotopy type of $Y$ and the rational
  cohomology of $X$.  This conjecture was disproved in a recent paper of the
  author and Weinberger; we gave an example where the ``predicted'' growth is
  $L^8$ but the true growth is $L^8\log L$.  Here we show, via a different
  mechanism, that the universe of possible such growth functions is quite large.
  In particular, for every rational number $r \geq 4$, there is a pair $X,Y$ for
  which the growth of $[X,Y]$ is essentially $L^r$.
\end{abstract}
\maketitle

\section{Introduction}

Let $X$ and $Y$ be compact piecewise Riemannian spaces, for example Riemannian
manifolds with boundary or finite simplicial complexes with a simplexwise linear
metric.  The \emph{growth} of the mapping class set $[X,Y]$ is the function
$g_{[X,Y]}:\mathbb{R}^+ \to \mathbb{N}$ given by
$$g_{[X,Y]}(L)=\#\{\alpha \in [X,Y]
\text{ which have a representative with Lipschitz constant}\leq L\}.$$
Asymptotically, this is a topological and even homotopy-theoretic notion; this is
because any compact, homotopy equivalent $X$ and $X'$ of this form are in fact
Lipschitz homotopy equivalent.

One well-studied class of examples is the growth of a finitely presented group
$\Gamma$, which can be thought of as the growth of the set of (based) homotopy
classes of maps $S^1 \to Y$, for any complex $Y$ with fundamental group $\Gamma$.
In that setting, Gromov's famous polynomial growth theorem \cite{GrPG} states
that growth is either asymptotic to a polynomial (for virtually nilpotent groups)
or superpolynomial (for all other groups).  In this paper, we instead concern
ourselves with simply connected $Y$ and general $X$, where the picture turns out
to be very different.

This notion of growth was first introduced by Gromov in \cite{GrHED}, where he
noted several facts about it.  First, he showed that the growth of $[S^n,S^n]$ is
$\Theta(L^n)$.  The upper bound is essentially homological: the domain of an
$L$-Lipschitz map can fit at most $O(L^n)$ preimages of a regular point, counted
with sign.  (We later give the simplest rigorous formulation of this argument,
which uses differential forms.)  Similar upper bounds can be found for the
pullback of any cohomology class.

Next, Gromov showed that the growth of $[S^{4n-1},S^{2n}]$ is $\Theta(L^{4n})$,
rather than the $\Theta(L^{4n-1})$ one might expect by generalizing from the
previous estimate.  The upper bound here generalizes to show that $g_{[S^n,Y]}$ is
bounded by a polynomial when $Y$ is simply connected.  (We are nowhere close to
showing whether this bound is sharp.)

In \cite[Ch.~7, p.~358]{GrMS}\footnote{The same range of ideas is discussed in
  Gromov's conference paper \cite{GrQHT}.}, Gromov gave some further conjectures
about $g_{[X,Y]}$ in the case where $Y$ is simply-connected; notably, he
conjectured that it is always asymptotic to $d^\alpha$, where $\alpha$ is an
integer determined by calculations in rational homotopy theory.

The heuristic that leads to this conjecture is as follows.  By obstruction
theory, a map $f:X \to Y$ is determined, non-uniquely and up to some finite
indeterminacy, by a finite number of ($\mathbb{Q}$-valued) rational homotopy
invariants.  Gromov sketches an argument that these invariants are always bounded
by a polynomial in the Lipschitz constant, and therefore (informally speaking)
all homotopy classes of $L$-Lipschitz maps live in a region of polynomial volume
in the set of rational homotopy classes.  If one makes a pair of very strong
assumptions, informally stated as follows:
\begin{enumerate}[(S1)]
\item integral homotopy classes are sprinkled evenly through this region
\item and all integral classes in the region are realizable via $L$-Lipschitz
  maps,
\end{enumerate}
then the conjecture holds.

In \cite{IRMC}, the author and Weinberger showed that $g_{[X,Y]}$ is bounded by a
polynomial whenever $Y$ is simply connected.  However, we disproved the stronger
conjecture by showing that
$$g_{[(S^3 \times S^4)^{\#2},S^4]}(L)=\Theta(L^8\log L).$$
This construction relied on violating assumption (S1): for this example, the
``density'' of integral homotopy classes in the rational ones grows without bound
as rational invariants increase.

In this paper we show that assumption (S2)---that all homotopy classes
whose rational homotopy invariants satisfy the obvious bounds are realizable with
Lipschitz constant $O(L)$---is also false in general.  This lets us produce a
much larger variety of growth functions of mapping class sets.  In fact, we prove
the following:
\begin{thmA}
  For every rational number $r>4$, there is a pair of simply connected spaces
  $X$ and $Y$ such that $\lim_{L \to \infty} \log_L(g_{[X,Y]}(L))=r$.
\end{thmA}
In other words, for every sufficiently large rational $r$ we have a growth
function which is more like $L^r$ than $L^{r'}$ for any other $r' \in \mathbb{Q}$.
A plausible conjecture about Lipschitz nullhomotopies implies that this growth
function is in fact $\Theta(L^r)$, but for now we cannot exclude the possibility
that it is slightly slower.  In addition, we give another example whose growth
function includes a logarithm.

\subsection*{Techniques}
All of these examples are constructed in essentially the same way.  The space $X$
consists of a wedge of two spheres $S^a \vee S^b$ of different dimensions,
together with a single third cell of some high dimension $n$; the space $Y$ is
obtained from $X$ by gluing on additional $(n+1)$-cells, killing $\pi_n(X)$.
A cellular map $X \to Y$ is determined (up to a bounded indeterminacy) by its
degrees on the $a$- and $b$-cells; this in turn determines the degree on the
higher-dimensional cell, putting an extra constraint on what the two invariants
can be for an $L$-Lipschitz map.

Another way of seeing this is via Lipschitz nullhomotopies.  Any map
$f:S^a \vee S^b \to Y$ (which is determined by two obstructions in $\mathbb{Z}$)
extends uniquely up to homotopy to $X$, via a nullhomotopy of $f \circ \partial$,
where $\partial$ is the attaching map of the top cell.  However, this extension
may be forced to have a much larger Lipschitz constant.  In other words, while
there is no absolute obstruction to performing the extension, there is a
quantitative obstruction in the form of a relatively long finite bar in the
persistence homotopy of $\pi_{n-1}(Y)$.

This phenomenon is closely related to another in quantitative homotopy theory:
the \emph{distortion} of homotopy group elements.  The distortion function of an
element $\alpha \in \pi_n(Y)$ is
$$\delta_\alpha(k)=\inf\{\Lip f \mid f:S^n \to Y, [f]=k\alpha\}.$$
In the given examples, the attaching map $S^{n-1} \xrightarrow{\partial}
S^3 \vee S^4$ of the top cell has a slow-growing distortion function; in other
words, for large $N \in \mathbb{Z}$, there is a relatively small map
$S^{n-1} \to S^3 \vee S^4$ whose homotopy class is $N[\partial]$.  However, in
order to nullhomotope this map in $Y$, one still needs a homotopy which has
relative degree $N$ over the $n$-cell; this means that the Lipschitz constant of
the homotopy is forced to be at least $\sim N^{1/n}$.

\subsection*{Outline of the paper}
Section 2 presents various background information needed for the main
constructions.  We invite the reader to skip it and refer back to it as needed.
The zoo advertised in the title is presented in Section 3.  Section 4 discusses
some situations in which the assumption (S2) holds and therefore Gromov's
prediction is closer to correct.

\section{Technical background}

\subsection{Cellular maps} \label{S:cell}

To enumerate homotopy classes between CW complexes, it is enough to consider
cellular maps.  The same turns out to be the case if we want to enumerate those
with an $L$-Lipschitz representative (up to a constant depending on the metrics
on $X$ and $Y$).
\begin{lem} \label{lem:cel}
  Let $X$ and $Y$ be piecewise Riemannian finite CW complexes with Lipschitz
  attaching maps.  Then any $L$-Lipschitz map $X \to Y$ can be deformed to a
  cellular one whose Lipschitz constant is $C(X,Y)(L+1)$.
\end{lem}
Thus if every cellular representative of a homotopy class is $\geq L$-Lipschitz,
then every representative is $\geq (L-1)/C$-Lipschitz.
\begin{proof}
  We can convert $X$ and $Y$ into homotopy equivalent simplicial complexes $X'$
  and $Y'$ skeleton-by-skeleton: at each stage, the attaching maps of cells in
  the next dimension are homotopic to simplicial maps on some triangulation of
  the boundary of the cell.  In particular, if we put the standard metric on each
  simplex of $X'$ and $Y'$, we get Lipschitz homotopy equivalences in both
  directions which we can also ensure are cellular.

  Now, given an $L$-Lipschitz map $f:X \to Y$, we can pre- and postcompose to get
  a $C_1L$-Lipschitz map $X' \to Y'$.  By the Lipschitz simplicial approximation
  theorem given in \cite{CDMW}, this map is homotopy equivalent to a simplicial
  map on a suitable subdivision of $X'$ which is $C_2(C_1L+1)$-Lipschitz.
  Finally we pre- and postcompose to get a cellular map $X \to Y$.
\end{proof}
So if we want an upper bound on the growth of $[X,Y]$, it is enough to obtain an
upper bound on the number of homotopy classes of $L$-Lipschitz cellular maps in
terms of $L$.  One way of obtaining such upper bounds is via degree
considerations:
\begin{lem}
  Suppose that $f:X \to Y$ is a cellular $L$-Lipschitz map.  Then for any
  $n$-cell $\sigma$ of $X$ and $n$-cell $\tau$ of $Y$, the induced map
  $$(\sigma,\partial \sigma) \to (Y^{(n)},Y^{(n)} \setminus \intr(\tau))$$
  has relative degree $O(L^n)$.
\end{lem}
\begin{proof}
  Since $f$ can be approximated by an $(L+\epsi)$-Lipschitz smooth map, we can
  assume that it is smooth.  Let $\omega$ be an $n$-form supported on the
  interior of $\tau$ with $\int_\tau \omega=1$ and
  $\lVert \omega \rVert_\infty \leq C$.  (Here the $\infty$-norm of a form is the
  supremum over its values on all frames of unit vectors.)  Then
  $$\left\lvert \int_\sigma f^*\omega \right\rvert
  \leq \vol\sigma \lVert f^*\omega \rVert_\infty
  \leq C\vol\sigma\lvert \Lip f \rvert^n.$$
  But this integral is the relative degree we are looking for.
\end{proof}

\subsection{Nullhomotopy and extension}

Suppose that we are constructing a cellular map $X \to Y$ skeleton-by-skeleton.
We would like to extend a map $f_{n-1}:X^{(n-1)} \to Y$ to the $n$-cells of $X$.
What restrictions are there on the degrees of the map on the $n$-cells?

Let $\alpha \in \pi_{n-1}(X^{(n-1)})$ be the attaching map of an $n$-cell, and
consider the pair of exact sequences
$$\xymatrix{
  \pi_n(X) \ar[r] \ar[d]^{f_*} & \pi_n(X,X^{(n-1)}) \ar[r] \ar[d]^{f_*} &
  \pi_{n-1}(X^{(n-1)}) \ar[d]^{(f_{n-1})_*} \\
  \pi_n(Y) \ar[r] & \pi_n(Y,Y^{(n-1)}) \ar[r] & \pi_{n-1}(Y^{(n-1)}).
}$$
Since $\pi_n(Y,Y^{(n-1)}) \cong H_n(Y)$, the middle arrow controls the degrees we
are interested in; since the diagram commutes, they are determined up to the
Hurewicz image of $\pi_n(Y)$.  For example, if $X$ and $Y$ each have a single
$n$-cell which kills a rationally nontrivial element of $\pi_{n-1}$, then the
Hurewicz image is trivial and the degree on that cell is determined by the
homotopy type of $f_{n-1}$.  This will be used in the construction in Section 3.

This observation was used in \cite{CDMW} to show that certain homotopically
trivial maps are hard to nullhomotope.  For example, let $Z_4$ be the space
obtained from $S^2 \vee S^2$ by gluing on two $5$-cells killing
$\pi_4(S^2 \vee S^2) \otimes \mathbb{Q}$.  While $\pi_4(Z_4)$ is trivial, an
$L$-Lipschitz map $S^4 \to Z_4$ may only be nullhomotopic via an
$\Omega(L^{6/5})$-Lipschitz homotopy.  The argument showing this is very similar
in flavor to the arguments in this paper.


\subsection{Smoothly formal spaces}

We say that a piecewise Riemannian space $Y$ is \emph{smoothly formal} if there
is a splitting algebra homomorphism $H^*(Y;\mathbb{Q}) \to \Omega^*(Y)$.  Note
that a smoothly formal space is in particular formal in the sense of Sullivan.
The only examples of smoothly formal spaces we are aware of are wedges of
Riemannian symmetric spaces.

The quantitative homotopy theory of such spaces is particularly nice.  In
particular, we have the following theorem, contrasting with the observation about
the space $Z_4$ in the previous subsection.
\begin{thm}[see \S5.3 of \cite{PCDF}] \label{thm:sym}
  If $X$ and $Y$ are finite complexes and $Y$ has the rational homotopy type of a
  simply connected smoothly formal space, then every nullhomotopic $L$-Lipschitz
  map $X \to Y$ has an $O(L\exp(\kappa(X,Y)\sqrt{\log L}))$-Lipschitz
  nullhomotopy.  In particular, for $X=S^n$ (or any other co-H-space), any two
  homotopic $L$-Lipschitz maps $X \to Y$ are homotopic via a
  $O(L\exp(\kappa(X,Y)\sqrt{\log L}))$-Lipschitz homotopy.
\end{thm}
This known bound is $o(L^{1+\epsi})$ for every $\epsi>0$.  In fact, I believe that
a linear bound can be attained.  This conjecture is recorded here as its truth
would significantly strengthen the results of this paper.
\begin{conj} \label{conj:sym}
  If $X$ and $Y$ are finite complexes and $Y$ has the rational homotopy type of a
  simply connected smoothly formal space, then every nullhomotopic $L$-Lipschitz
  map $X \to Y$ has an $O(L)$-Lipschitz nullhomotopy.
\end{conj}

\subsection{Wedges of spheres}

An important class of smoothly formal spaces is wedges of spheres.  Here we
describe the homotopy theory of the space $\bigvee_i S^{n_i}$.

Let $X$ be any space.  Given maps $f:S^i \to X$ and $g:S^j \to X$, their
\emph{Whitehead product} $[f,g]:S^{i+j-1} \to X$ is the composition
$$S^{i+j-1} \xrightarrow{\text{attaching map of the $(i+j)$-cell of }
  S^i \times S^j} S^i \vee S^j \xrightarrow{f \vee g} X.$$
Note that while we defined the Whitehead product of two maps, the Whitehead
product is also well-defined on the level of homotopy groups.  (We will use the
same notation for the map- and homotopy-level products.)  In fact, up to
homotopy, it is bilinear, graded commutative, and satisfies the graded Jacobi
identity
$$(-1)^{ij}[f,[g,h]]+(-1)^{jk}[g,[h,f]]+(-1)^{ki}[h,[f,g]]=0,$$
for $f:S^i \to X$, $g:S^j \to X$, and $h:S^k \to X$.  In other words, the
Whitehead product defines a graded Lie bracket on the graded module
$\pi_{*+1}(X)$ (i.e.~elements of $\pi_2(X)$ have degree 1, etc.)

For $X=\bigvee_i S^{n_i}$, the rationalization $\pi_{*+1}(X) \otimes \mathbb{Q}$ is
in fact a \emph{free} graded Lie algebra over $\mathbb{Q}$ generated by an
$(n_i-1)$-dimensional generator for each $n_i$.

Given an ordering on the generators of a free Lie algebra over a field of
characteristic zero, one obtains a \emph{Hall set} of monomials which form a
basis for the Lie algebra as a vector space.
It is not hard to see that this Hall set, together with $[a,a]$ for every element
$a$ of odd degree in the Hall set, still forms a basis for the free \emph{graded}
Lie algebra; cf.~\cite{Hilton}, where this is shown for the homotopy algebra of a
wedge of spheres of the same dimension.  We follow Hilton's formulation of the
Hall set, which differs slightly from some more modern treatments.

\section{The zoo}

In this section we give a large number of simple examples of spaces $X$ and $Y$
for which the growth of $[X,Y]$ can be determined enough to rule out polynomial
behavior.  Each is given by a few cells with attaching maps given by iterated
Whitehead products.
\begin{ex} \label{ex1}
  Let $X=(S^3 \vee S^4) \cup_{[\id_{S^3},[\id_{S^3},\id_{S^4}]]} D^9$, and let $Y$ be the
  space obtained from $X$ by adding $10$-cells to kill $\pi_9(X)$.  Then the
  growth of $[X,Y]$ is $O(L^{13/2})$ and $\omega(L^{13/2-\epsi})$ for every
  $\epsi>0$.
\end{ex}
Note that purely obstruction-theoretic concerns would suggest that there
should be $\Theta(L^7)$ homotopy classes of $L$-Lipschitz maps in this case as
well as that of Example \ref{ex2}.  Indeed, in both cases the growth turns out to
be slower.

If we assume Conjecture \ref{conj:sym}, the $O(L^{13/2})$ bound is sharp by the
same argument using the conjecture instead of Theorem \ref{thm:sym}.
\begin{proof}
  By the discussion in \S\ref{S:cell}, we can assume that maps $X \to Y$ are
  cellular.  By obstruction theory, since $H^k(X;\pi_k(Y))$ is only nonzero in
  degrees $3$ and $4$, the homotopy class of such a map is determined by the
  degrees of the restriction-projections $\alpha:S^3 \to S^3$ and
  $\beta:S^4 \to S^4$ and a $\mathbb{Z}/2\mathbb{Z}$ invariant $\tau$ coming from
  $\pi_4(S^3)$.  In turn, the two degrees determine the degree on the $9$-cell of
  $X$: the boundary of this cell is mapped to $S^3 \vee S^4 \subset Y$ via
  $[\alpha,[\alpha,\beta \vee \tau]]$, and the bilinearity of the Whitehead
  product means that any nullhomotopy of this boundary map in $Y$ must have
  degree $(\deg\alpha)^2\deg\beta$.

  Assume to simplify notation that the metric on $Y$ extends that on $X$.  A
  cellular map with Lipschitz constant $\leq L$ must then satisfy
  $$\left\{
  \begin{aligned}
    \lvert\deg\alpha\rvert &\leq L^3 \\
    \lvert\deg\beta\rvert &\leq L^4 \\
    \lvert(\deg\alpha)^2\deg\beta\rvert &\leq L^9.
  \end{aligned}\right.$$
  An integral shows that the area of the region in $\mathbb{R}^2$ satisfying
  these conditions is $O(L^{13/2})$; the number of lattice points is similar.
  This demonstrates the upper bound.

  For the lower bound, we need to actually exhibit $\Omega(L^{13/2-\epsi})$
  different maps whose Lipschitz constant is $O(L)$, for every $\epsi>0$.  So fix
  $\epsi$, and consider any $a, b \in \mathbb{Z}$ satisfying
  \begin{equation} \label{bounds} \left\{
    \begin{aligned}
      \lvert a \rvert &\leq L^{3-\epsi} \\
      \lvert b \rvert &\leq L^{4-\epsi} \\
      \lvert a^2b \rvert &\leq L^{9-2\epsi}.
    \end{aligned}\right.
  \end{equation}
  We will construct a cellular $O(L)$-Lipschitz map $f_{a,b}:X \to Y$ whose
  degrees on the $3$- and $4$-cell are $a$ and $b$, respectively.  Write
  $u_{k,d}:S^k \to S^k$ for a maximally efficient, hence $O(d^{1/k})$-Lipschitz map
  of degree $d$.  We start by setting $f_{a,b}$ to be $u_{3,a}$ on the 3-cell and
  $u_{4,b}$ on the 4-cell.  All that is left is to extend to the 9-cell.

  As a shorthand, given a map $f:S^k \to S^k$, we write $d \cdot f$ to mean
  $f \circ u_{k,d}$ (a particular efficient representative of the homotopy class
  $d[f]$.)  Then let $g_1:S^8 \to S^3 \vee S^4$ be given by
  $[u_{3,a},[u_{3,a},u_{4,b}]]$ and $g_2:S^8 \to S^3 \vee S^4$ be given by
  \begin{equation} \label{g2}
    S^8 \xrightarrow{\text{pinch the equator}} S^8 \vee S^8
    \xrightarrow{\left[s \cdot \id_{S^3},t \cdot [\id_{S^3},\id_{S^4}]\right]
      \vee e \cdot \left[\id_{S^3},[\id_{S^3},\id_{S^4}]\right]} S^3 \vee S^4,
  \end{equation}
  where $s=\lfloor L^{3-\epsi} \rfloor$, $t=\lfloor L^{-(3-\epsi)}a^2b \rfloor$, and
  $e=a^2b-st$ is an $O(L^6)$ correction term that makes $g_1$ and $g_2$
  homotopic.  Note that $g_1$ and $g_2$ are both $O(L^{1-\epsi/4})$-Lipschitz.
  Then by Theorem \ref{thm:sym}, there is an $O(L)$-Lipschitz homotopy between
  them in $S^3 \vee S^4$.

  Moreover, $g_2$ extends to an $O(L^{1-\epsi/4})$-Lipschitz map $D^9 \to Y$.  We
  demonstrate this by giving separate nullhomotopies of
  $$\left[s \cdot \id_{S^3},t \cdot [\id_{S^3},\id_{S^4}]\right]\quad\text{and}\quad
  e \cdot \left[\id_{S^3},[\id_{S^3},\id_{S^4}]\right].$$
  The second of these maps has an $O(L^{3/4})$-Lipschitz nullhomotopy given by
  composing the attaching map $D^9 \to Y$ of the 9-cell with the cone
  $Cu_{8,e}:D^9 \to D^9$ of $u_{8,e}$.  The first map has an
  $O(L^{1-\epsi/4})$-Lipschitz nullhomotopy factoring through a sequence of maps
  $$D^9 \xrightarrow{\text{top cell}} S^3 \times S^6
  \xrightarrow{u_{3,s} \times u_{6,t}} S^3 \times S^6 \to X \subset Y;$$
  here the last arrow sends $* \times S^6$ to $S^3 \vee S^4$ via the Whitehead
  product.

  Finally, we extend $f_{a,b}$ to the $9$-cell of $X$ by concatenating the
  homotopy from $g_1$ to $g_2$ and the nullhomotopy of $g_2$.
\end{proof}
We now give the more general version.  The proof is abbreviated since it is very
similar to the one above once $X$ and $Y$ are constructed.
\begin{thm} \label{thm:ex1}
  For any rational number $r>4$, there are spaces $X$ and $Y$ such that the
  growth of $[X,Y]$ is $O(L^r)$ and $\omega(L^{r-\epsi})$ for any $\epsi>0$.
\end{thm}
\begin{proof}
  We fix natural numbers $\ell<m$ and $1 \leq p<q$ such that
  $$r=\ell+m+\frac{2-p-q}{q}.$$
  Note that every rational number in $(\ell+m-2,\ell+m-1]$ can be written this
  way for some $p$ and $q$.  In addition we set $\ell$ and $m$ so that
  $\ell \geq 2$ and $m \geq 4$ and $m-\ell$ is $1$ or $2$.  Fix
  $$n=p(\ell-1)+q(m-1)+2.$$

  Let $\zeta \in \pi_{n-1}(S^\ell \vee S^m)$ be given by the iterated Whitehead
  product
  $$\zeta=[\underbrace{\id_{S^\ell},[\cdots[\id_{S^\ell}}_{p-1\text{ times}},
        [\underbrace{\id_{S^m},[\cdots[\id_{S^m}}_{q\text{ times}},
              \id_{S^\ell}]\cdots]]]\cdots]].$$
  This is an element of infinite order since, for the ordering
  $\id_{S^m}<\id_{S^\ell}$, it is a Hall element of the free Lie algebra.  We let
  $X=(S^\ell \vee S^m) \cup_\zeta D^n$ and $Y$, like before, is obtained from $X$
  by adding $(n+1)$-cells to kill $\pi_n(X)$.  Then a cellular map $X \to Y$ with
  Lipschitz constant $L$ is determined (up to an element of the finite group
  $\pi_m(S^\ell)$) by the degrees $a$ and $b$ of the restrictions to
  $S^\ell \to S^\ell$ and $S^m \to S^m$.  These degrees must satisfy
  \[\left\{
  \begin{aligned}
    \lvert a \rvert &\leq L^\ell \\
    \lvert b \rvert &\leq L^m \\
    \lvert a^p b^q \rvert &\leq L^n.
  \end{aligned}\right.\]
  The volume of this region in $\mathbb{R}^2$ is $\Theta(L^r)$.

  As before, for any $\epsi>0$, we can actually realize $\Omega(L^{r-\epsi})$
  elements which satisfy corresponding slightly stronger bounds.  To do this, we
  take a map $S^\ell \vee S^m \to Y$ with degrees $a$, $b$ satisfying such bounds
  and extend it to the $n$-cell of $X$ via a concatenation of two homotopies: one
  from the induced map on the boundary of this cell to a Whitehead product
  $$\bigl[\lfloor L^{\ell-\epsi} \rfloor \cdot \id_{S^\ell},
    \lfloor L^{-(\ell-\epsi)}a^pb^q \rfloor \cdot
     [\underbrace{\id_{S^\ell},[\cdots[\id_{S^\ell}}_{p-2\text{ times}},
     [\underbrace{\id_{S^m},[\cdots[\id_{S^m}}_{q\text{ times}},\id_{S^\ell}]\cdots]]]
         \cdots ]]\bigr]$$
  together with an $O(L^{n-\ell})$ correction term, followed by a nullhomotopy of
  that map which factors through $S^\ell \times S^m$.
\end{proof}
\begin{ex} \label{ex2}
  Let $X=(S^3 \vee S^4) \cup_{[[\id_{S^3},\id_{S^4}],[\id_{S^3},\id_{S^4}]]} D^{12}$, and let
  $Y$ be the space obtained from $X$ by adding $13$-cells killing $\pi_{12}(X)$.
  Then the growth of $[X,Y]$ is $\Theta(L^6\log L)$.
\end{ex}
\begin{proof}
  This example is in many ways similar to the previous ones.  Once again, an
  $L$-Lipschitz map $X \to Y$ is determined up to a finite kernel by two degrees
  $a$ and $b$ which satisfy
  \begin{equation} \label{bounds2} \left\{
    \begin{aligned}
      \lvert a \rvert &\leq L^3 \\
      \lvert b \rvert &\leq L^4 \\
      \lvert a^2b^2 \rvert &\leq L^{12};
    \end{aligned}\right.
  \end{equation}
  integrating gives $\Theta(L^6\log L)$ such pairs.

  However, in this case we can explicitly construct an $O(L)$-Lipschitz map for
  every such pair of degrees $(a,b)$, using a refinement of the construction in
  Theorem \ref{thm:ex1}.  On the $3$- and $4$-cells, we fix explicit
  $O(L)$-Lipschitz maps $w_a:S^3 \to S^3$ and $w_b:S^4 \to S^4$ of the correct
  degree.  Then on the $12$-cell, the map is a concatenation of two homotopies
  which together form a nullhomotopy of the Whitehead square
  $[[w_a,w_b],[w_a,w_b]]$.

  The first homotopy is a fiberwise Whitehead square $(h_1)_t=[g_t,g_t]$, where
  $g:S^6\times[0,1] \to S^3 \vee S^4$ is an $O(L)$-Lipschitz homotopy
  (constructed explicitly later) which takes
  $$[w_a,w_b] \simeq_{g_t} ab \cdot [\id_{S^3},\id_{S^4}].$$
  Note that the map at time $1$ is $O(L)$-Lipschitz because $ab \leq L^6$; in the
  analogous step in Example \ref{ex1}, where $ab$ may be greater, we needed to
  exploit the bilinearity of both Whitehead products simultaneously, making an
  explicit construction harder to obtain.

  We concatenate this with a nullhomotopy of
  $$[ab \cdot [\id_{S^3},\id_{S^4}],ab \cdot [\id_{S^3},\id_{S^4}]]$$
  which proceeds via the ($O(L)$-Lipschitz) composition
  $$D^{12} \xrightarrow{\text{attaching map}} S^6 \times S^6
  \xrightarrow{u_{6,ab} \times u_{6,ab}} S^6 \times S^6 \to X \subset Y,$$
  where the last arrow takes each of the $6$-cells to $S^3 \vee S^4$ via the
  Whitehead product.

  The key part of the proof is specifying the maps $w_a$, $w_b$, and $g$ more
  precisely in order to ensure that everything is $O(L)$-Lipschitz.  If we knew
  Conjecture \ref{conj:sym}, we would obtain the existence of such a homotopy
  automatically, but this case is simple enough that we can do it by hand.
  Roughly speaking, we keep track of the preimages of the north pole of $S^3$ (a
  3-manifold) and the north pole of $S^4$ (a 2-manifold).  Each preimage starts
  out as a grid of spheres, and these grids are linked together; gradually, we
  tease the links apart, so that at the end the preimages consist of $\leq L^6$
  pairs of linked $S^2$ and $S^3$.  So that this actually extends to an
  $O(L)$-Lipschitz homotopy, we make sure that the bordisms have non-overlapping
  embedded normal neighborhoods of radius $\Omega(1/L)$.  We now give the
  construction in more detail; the bordisms themselves are laid out in Figure
  \ref{fig:htpy}.

  We first specify $w_a$.  We tile the bottom of a unit $3$-cube with $a$
  identical open balls of diameter $1/L$, $L$ to a side.  We get
  $\lceil a/L^2 \rceil$ layers of such balls, with the last layer incomplete.
  Fix an embedding of this cube in $S^3$; we let the map $w_a:S^3 \to S^3$ send
  every point outside the balls inside the cube to the basepoint $*$, and each of
  the balls homeomorphically to $S^3 \setminus *$.  We construct $w_b$
  analogously.  Note that
  $$\left\lceil\frac{a}{L^2}\right\rceil\cdot\Bigl\lceil\frac{b}{L^3}\Bigr\rceil
  \leq 2L.$$

  We now describe the homotopy $g$; this is most easily done replacing the
  dimensions $3$ and $4$ by some general $i,j \geq 2$, and $6$ by $i+j-1$.  We
  give a pictorial step-by-step construction of $g$ which may be taken literally
  when $i,j=2$ but is easily seen to work for other $i,j$.

  The domain $S^{i+j-1}$ of $[w_a,w_b]$ can be split into two solid tori: the
  preimage $D^i \times S^{j-1}$ of $S^i$, on which the map is $w_a$ after
  projection onto the first coordinate, and the preimage $S^{i-1} \times D^j$ of
  $S^j$ on which the map is $w_b$ after projection onto the second coordinate.

  Within the first torus, we have a $[0,1]^i \times S^{j-1}$, the subspace on
  which $w_a$ is nontrivial.  Let's say that the last coordinate represents
  ``depth''; then we complete this $[0,1]^{i-1} \times ([0,1] \times S^{j-1})$ to
  a cylinder $[0,1]^{i-1} \times D^j$ where the map on each filling $* \times D^j$
  restricts to $w_b$.  We then slice this cylinder into $L^{i-1}$ thickened
  $j$-disks which we call ``pancakes''; the case $i=j=2$ is shown in Figure
  \ref{fig:cake}.
  \begin{figure}
    \centering
    \begin{tikzpicture}
      \draw[very thick,color=blue] (0,0) circle (4 and 1.3);
      \draw[very thick,color=blue] (0,0) circle (3.5 and 1.1);
      \foreach \x in {1,...,5}
        \draw[very thick,color=blue] (0,-0.4*\x)+(180:4) arc (180:360:4 and 1.3);
      \foreach \x in {0,1,2} {
        \foreach \y in {0,1,...,5} {
          \fill[color=red] (-0.5+1.2*\x-0.25*\y,-0.5+0.2*\y) circle (0.07);
          \draw[very thick,color=red,-fast cap] (-0.5+1.2*\x-0.25*\y,-0.5+0.2*\y)
              -- (-0.5+1.2*\x-0.25*\y,-0.8+0.1*\y);
        }
      }
      \draw[<->] (180:4.3) -- (180:3.2) node[pos=0.5,anchor=south] {$s$ preimages};
      \draw[<->] (0,0.2)++(0:4.5) -- ++(0,-2.4) node[pos=0.5,anchor=west] {$L$ preimages};
    \end{tikzpicture}
    \caption{The chosen cylinder $S^{i-1} \times D^j$ when $i=j=2$; in this case
      the pancakes are horizontal slices, arranged in a stack.  The blue and red
      lines are preimages of the north pole of $S^i$ and $S^j$, respectively; the
      red ones wrap all the way around, outside this chosen subdomain.}
    \label{fig:cake}
  \end{figure}
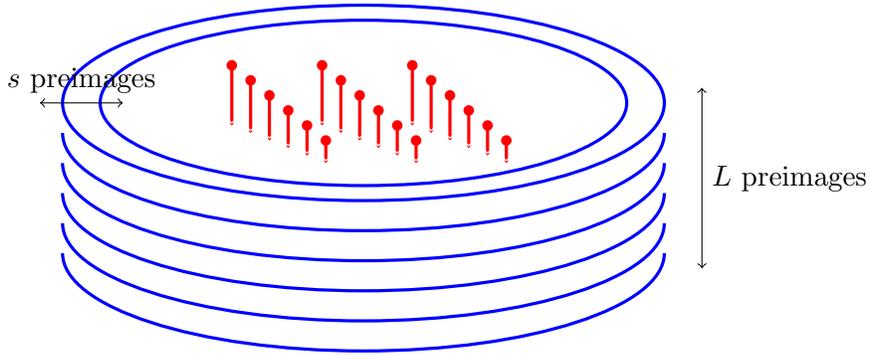
  Let $s=\lceil a/L^{i-1} \rceil$ and $t=\lceil b/L^{j-1} \rceil$ be the depths to
  which we filled the two cubes to build $w_a$ and $w_b$; then each pancake
  contains about $s$ strands of the preimage of $S^i$.

  We now give the construction of the previously mentioned bordism; we do this
  the same way within each pancake, as shown in Figure \ref{fig:htpy}.
  \begin{figure}
    \centering
    \begin{subfigure}[b]{0.47\textwidth}
      \begin{tikzpicture}
        \draw[very thick,color=blue,rounded corners=10pt]
          (0,4) -- (5,4) -- (7,0) -- (2,0) -- cycle;
        \draw[very thick,color=blue,rounded corners=10pt]
          (0.7,3.6) -- (4.7,3.6) -- (6.3,0.4) -- (2.3,0.4) -- cycle;
        \foreach \x in {0,1,2}
          \foreach \y in {0,1,...,5}
            \draw[fast cap-fast cap,very thick,color=red]
              (3+1.2*\x-0.25*\y,0.5+0.5*\y) -- (3+1.2*\x-0.25*\y,1+0.5*\y);
        \draw[<->] (1,1.33) -- (2.17,1.33) node[pos=0.4,anchor=south] {$s$};
        \draw[<->] (1.67,2.67) -- (4.67,2.67) node[pos=0.3,anchor=south] {$t$};
        \draw[<->] (3.4,3.33) -- (4.73,0.67) node[pos=0.5,anchor=west] {$L$};
      \end{tikzpicture}
      \caption{The initial preimages of the north poles of $S^i$ (blue) and $S^j$
        (red) inside a single pancake.}
    \end{subfigure}
    \quad
    \begin{subfigure}[b]{0.47\textwidth}
      \begin{tikzpicture}
        \draw[very thick,color=blue,rounded corners=10pt]
          (0,4) -- (5,4) -- (7,0) -- (2,0) -- cycle;
        \draw[very thick,color=blue,rounded corners=10pt]
          (0.7,3.6)
          {[rounded corners=2pt] -- (2.1,3.6) -- (2.3,3.4) -- (2.3,3.6) -- (3.3,3.6) -- (3.5,3.4) -- (3.5,3.6)}
          -- (4.7,3.6) -- (6.3,0.4)
          {[rounded corners=2pt] -- (5.1,0.4) -- (4.9,0.6) -- (4.9,0.4) -- (3.9,0.4) -- (3.7,0.6) -- (3.7,0.4)}
          -- (2.3,0.4) -- cycle;
        \foreach \x in {0,1,2}
          \foreach \y in {0,1,...,5}
            \draw[fast cap-fast cap,very thick,color=red]
              (3+1.2*\x-0.25*\y,0.5+0.5*\y) -- (3+1.2*\x-0.25*\y,1+0.5*\y);
        \draw[->,line width=2pt] (2.4,3.2) -- (2.7,2.6);
        \draw[->,line width=2pt] (3.6,3.2) -- (3.9,2.6);
        \draw[->,line width=2pt] (4.8,0.8) -- (4.5,1.4);
        \draw[->,line width=2pt] (3.6,0.8) -- (3.3,1.4);
      \end{tikzpicture}
      \caption{The first step is to push the preimages of $S^i$ inwards\ldots}
    \end{subfigure}
    \begin{subfigure}[b]{0.47\textwidth}
      \begin{tikzpicture}
        \foreach \x in {0,2,4} {
          \draw[very thick,color=blue,rounded corners=10pt]
            (\x,4) -- (\x+1.6,4) -- (\x+3.6,0) -- (\x+2,0) -- cycle;
          \draw[very thick,color=blue,rounded corners=10pt]
            (\x+0.6,3.6) -- (\x+1.4,3.6) -- (\x+3,0.4) -- (\x+2.2,0.4) -- cycle;
          \foreach \y in {0,1,...,5}
            \draw[fast cap-fast cap,very thick,color=red]
              (2.4+\x-0.25*\y,0.5+0.5*\y) -- (2.4+\x-0.25*\y,1+0.5*\y);
          \draw[->,line width=2pt] (\x,0) +(2.2,2) .. controls +(0.5,0.8) and +(-0.3,0.3) .. +(3,2.3);
        }
        \draw[<->] (0.2,4.3) -- (5.8,4.3) node[pos=0.5,anchor=south] {$(2s+1)t=O(L)$};
      \end{tikzpicture}
      \caption{\ldots separating the preimages of $S^j$ into layers, each
        encircled by preimages of $S^i$.  Then we take the skinny rings one by
        one\ldots}
    \end{subfigure}
    \quad
    \begin{subfigure}[b]{0.47\textwidth}
      \begin{tikzpicture}
        \foreach \x in {0,2,4} {
          \draw[very thick,color=blue,rounded corners=8pt]
            (\x,4) -- (\x+0.6,4) -- (\x+2.6,0) -- (\x+2,0) -- cycle;
          \draw[very thick,color=blue,rounded corners=6pt]
            (\x+1.3,3.6) -- (\x+1.7,3.6) -- (\x+3.3,0.4) -- (\x+2.9,0.4) -- cycle;
          \foreach \y in {0,1,...,5} {
            \draw[fast cap-fast cap,very thick,color=red]
              (1.9+\x-0.25*\y,0.5+0.5*\y) -- (1.9+\x-0.25*\y,1+0.5*\y);
            \draw[very thick,color=red] (\x+2.79-0.25*\y,0.7+0.5*\y) arc(-54:286:0.2 and 0.18);
          }
        }
      \end{tikzpicture}
      \caption{\ldots and bud them off, creating a new layer for each.}
    \end{subfigure}
    \begin{subfigure}[b]{0.47\textwidth}
      \begin{tikzpicture}
        \foreach \x in {0,2,4} {
          \draw[very thick,color=blue,rounded corners=6pt]
            (\x+0.5,3.6) -- (\x+0.9,3.6) -- (\x+2.5,0.4) -- (\x+2.1,0.4) -- cycle;
          \draw[very thick,color=blue,rounded corners=6pt]
            (\x+1.3,3.6) -- (\x+1.7,3.6) -- (\x+3.3,0.4) -- (\x+2.9,0.4) -- cycle;
          \foreach \y in {0,1,...,5} {
            \draw[fast cap-fast cap,very thick,color=red]
              (1.4+\x-0.25*\y,0.5+0.5*\y) -- (1.4+\x-0.25*\y,1+0.5*\y);
            \draw[very thick,color=red] (\x+1.99-0.25*\y,0.7+0.5*\y) arc(-54:286:0.2 and 0.18);
            \draw[very thick,color=red] (\x+2.79-0.25*\y,0.7+0.5*\y) arc(-54:286:0.2 and 0.18);
          }
        }
      \end{tikzpicture}
      \caption{Finally, we separate the $L^{j-1}$-part links into Hopf links, and
        simultaneously\ldots}
    \end{subfigure}
    \quad
    \begin{subfigure}[b]{0.47\textwidth}
      \begin{tikzpicture}
        \foreach \x in {-1,0,...,4} {
          \foreach \y in {0,1,...,5} {
            \draw[very thick,color=blue] (\x+2.94-0.25*\y,0.66+0.5*\y) arc(100:440:0.2 and 0.1);
            \draw[very thick,color=red] (\x+2.9-0.25*\y,0.5+0.5*\y) arc(-50:290:0.2 and 0.18);
          }
        }
      \end{tikzpicture}
      \caption{\ldots kill the big leftover preimages of $S^j$ which are no
        longer linked with anything.  In the end, we have a grid of Hopf links,
        as desired.}
    \end{subfigure}
    \caption{The stages of the homotopy $g$, as seen on the neighborhood of a
      $j$-dimensional cross-section.  The preimage of the north poles of $S^i$
      and $S^j$ are shown in blue and red, respectively, and the homotopy
      corresponds to a simultaneous bordism of these two manifolds.  The vertical
      direction stands in for a $(j-1)$-dimensional subspace.} \label{fig:htpy}
  \end{figure}
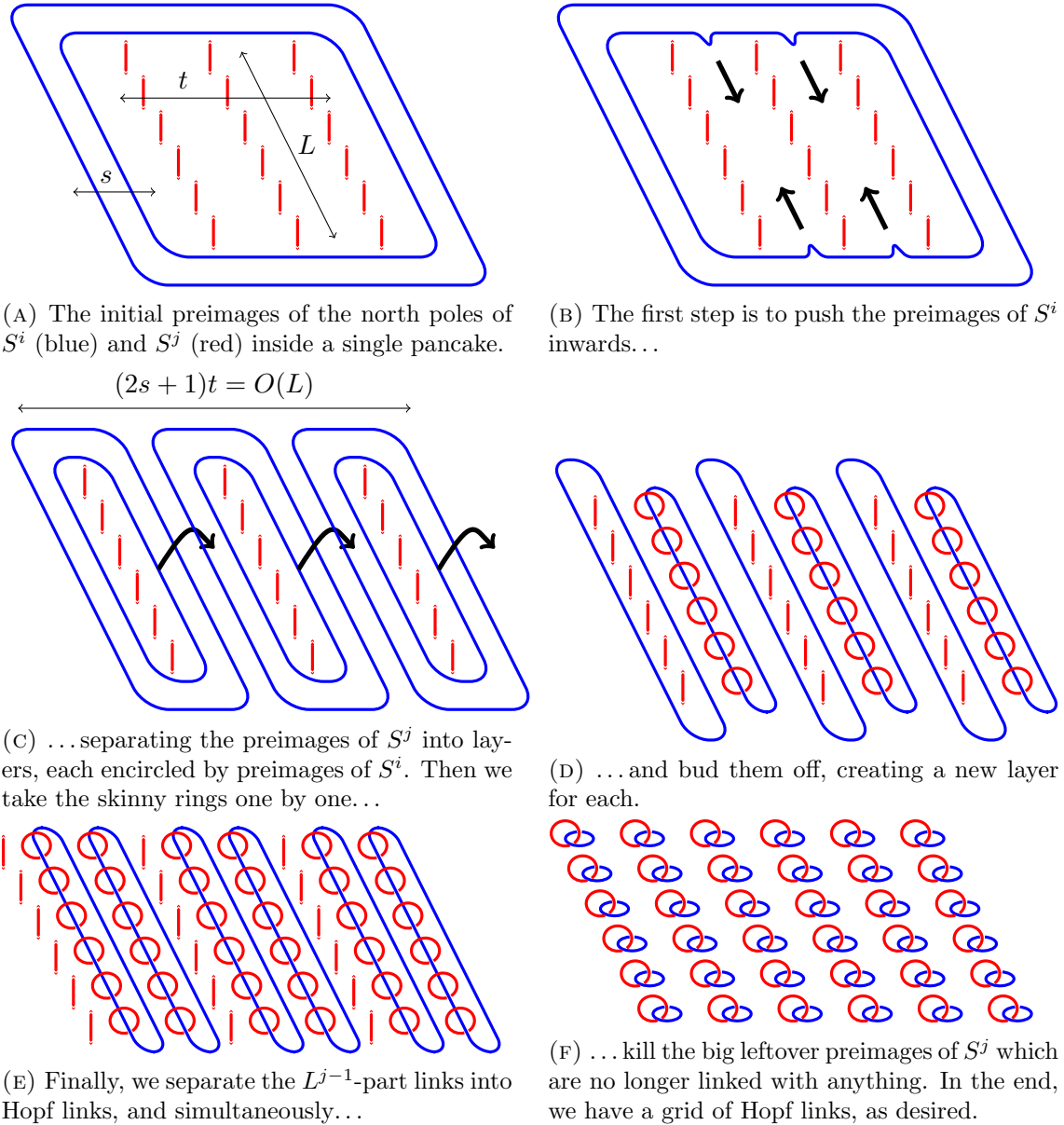
\end{proof}

\section{Symmetric spaces}

As shown in \cite{IRMC}, the growth of $[X,Y]$ still need not be polynomial even
when $Y$ is the smoothly formal space $S^4$.  However, in this case the
assumption (S2) given in the introduction can be formalized, and holds, or nearly
so.  We will show this here for symmetric spaces.  The argument easily extends to
wedges of symmetric spaces, but to give the correct statement there would require
some extra definitions which would take us too far afield.

In this section we take the naive, skeleton-by-skeleton approach to constructing
efficient maps.  For example, given a simplicial complex $X$ and an odd
$n \geq 3$, maps $X \to S^n$ are determined up to a finite indeterminacy by an
obstruction class in $H^n(X)$.  Given a simplicial cochain $w \in C^n(X)$
representing this homology class, we can construct a map as follows:
\begin{itemize}
\item Send the $(n-1)$-skeleton to the basepoint of the sphere.
\item Send each $n$-simplex to an optimally efficient map whose degree is
  dictated by the given cochain.
\item Extend to higher skeleta.  For each $k$-simplex, we have a finite number of
  choices up to homotopy in non-canonical bijection with $\pi_k(S^n)$.
\end{itemize}
The Lipschitz constant of the map on $X^{(n)}$ generated by the first two steps
is $C(n)(\lVert w \rVert_\infty)^{1/n}$.  However, to do the later steps in a
quantitative way, one also needs bounds on the size of the smallest map filling
in a given map on the boundary of a $k$-simplex; in other words, estimates on the
size of a nullhomotopy.  Luckily, Theorem B of \cite{CDMW} shows that one can
construct linear nullhomotopies in this situation, and therefore the Lipschitz
constant of the resulting map is linear in $(\lVert w \rVert_\infty)^{1/n}$.
Moreover, by Lemma \ref{lem:cel}, it is sharp up to a constant: every
$L$-Lipschitz map has a first obstruction $w$ with
$\lVert w \rVert_\infty \leq CL^n$.

Things get more complicated if we look at maps $X \to S^n$ for $n$ even.  Here
there is a second rational obstruction in dimension $2n-1$; more importantly, we
no longer have a linear bound on the sizes of nullhomotopies, unless we assume
Conjecture \ref{conj:sym}.  Nevertheless, we can follow the same procedure as
above up to the $(2n-2)$-skeleton, and then perform the following steps:
\begin{itemize}
\item Fill in the $(2n-1)$-simplices in the most efficient possible way.
\item Now there is an obstruction in $H^{2n-1}(X)$ to extending to our desired
  homotopy class.  Given a cochain $z$ representing this obstruction, we glue in
  maps $S^{2n-1} \to S^n$ on each simplex as dictated by $z$.  These maps can have
  Lipschitz constant $\leq C(n)(\lVert z \rVert_\infty)^{1/2n}$, as per
  \cite{GrHED}.
\item Finally we extend to higher skeleta as before.  However, this time our
  extensions are only guaranteed by Theorem \ref{thm:sym} to be
  $O(L^{1+\epsi})$-Lipschitz for any given $\epsi>0$, where $L$ is the Lipschitz
  constant of the map on the previous skeleton.
\end{itemize}
Thus as we go up, we accumulate nonlinearity.  However, in the end, we can still
ensure that the Lipschitz constant is at most
$$C(n,\dim X,\epsi)\bigl(\lVert w \rVert_\infty^{1/n}+\lVert z \rVert_{\infty}^{1/2n}
\bigr)^{1+\epsi}$$
for any $\epsi>0$ we choose.  The converse is a bit harder to state, since $z$ is
a relative obstruction.  Given an $L$-Lipschitz map, we again have
$\lVert w \rVert_\infty \leq C(n)L^n$; moreover, given two maps which coincide on
$X^{(2n-2)}$, the obstruction $z^\prime$ to homotoping them in dimension $2n-1$,
which is well-defined, satisfies $\lVert z^\prime \rVert_\infty \leq C(n)L^{2n}$.
Thus the total number of homotopy classes of $L$-Lipschitz maps is
$$O(L^{n\rk H^n(X)+2n\rk H^{2n-1}(X)}),$$
although it could be much smaller since different obstructions may specify the
same homotopy class.  The main theorem of this section uses the results of
\cite{PCDF} to generalize these observations to spaces with the rational homotopy
type of symmetric spaces.

Now, more generally, let $Y$ be a finite CW complex whose 1-skeleton is a point.
Then a homotopy class of maps $X \to Y$ is specified (in an obviously non-unique
way) by a sequence of obstructions in $C^k(X;\pi_k(Y))$, $k=2,\ldots,\dim X$.
More precisely, the $k$th such obstruction sends each $k$-simplex $p$ to an
element of the torsor for $\pi_k(Y)$ consisting of nullhomotopies of the map on
$\partial p$.  Thus it is only well-defined once the map is fixed on lower
skeleta.  Nevertheless, as in the case of even-dimensional spheres, we can still
compare the sizes of obstructions realizable with Lipschitz constant $L$.

\begin{thm} \label{thm:s4}
  Let $X$ and $Y$ be finite complexes such that $Y$ has the rational homotopy
  type of a symmetric space.  Then there is a way to make sense of the ``size''
  of a sequence of obstructions so that, assuming Conjecture \ref{conj:sym}, the
  minimal Lipschitz constant of a representative of $\alpha \in [X,Y]$ is equal
  up to a multiplicative constant to the minimal size of a sequence of
  obstructions describing it.  Without that assumption, they are related in at
  worst a slightly superlinear way.

  Specifically:
  \begin{itemize}
  \item Denote the rational Hurewicz map by $h_k:\pi_k(Y) \otimes \mathbb{Q} \to
    H_k(Y;\mathbb{Q})$.
  \item For every $k$, fix a splitting
    $$\pi_k(Y) \otimes \mathbb{Q}=\ker h_k \oplus N_k$$
    with projections $p_1,p_2$ sending elements to $\ker h_k$ and $N_k$
    respectively, and a norm $\lvert\cdot\rvert$ on
    $\pi_k(Y) \otimes \mathbb{Q})$.
  \end{itemize}
  Then there are a constant $c>0$ and a function
  $g:\mathbb{R}^+ \to \mathbb{R}^+$ depending on these norms, and a choice of
  zero for every torsor described above, such that:
  \begin{enumerate}[(i)]
  \item Every $cL$-Lipschitz map lies in a homotopy class which can be described
    by obstructions $\beta_k \in C^k(X;\pi_k(Y))$ such that
    $$\sup_{k\text{-cells $\sigma$ of }X} |p_1(\beta_k(\sigma))| \leq L^{k+1}
    \quad\text{and}\quad
    \sup_{k\text{-cells $\sigma$ of }X} |p_2(\beta_k(\sigma))| \leq L^k.$$
  \item Conversely, every homotopy class which can be described by such
    obstructions (even just rationally) has a $g(L)$-Lipschitz representative.
  \end{enumerate}
  Assuming Conjecture \ref{conj:sym}, this $g$ is linear.  Otherwise, we merely
  know that $g(L)=O(L^{1+\epsi})$ for every $\epsi>0$.
\end{thm}
\begin{proof}
  The proof uses the following corollary of \cite[Thm.~5.3]{PCDF}:
  \begin{thm} \label{thm:dist}
    Define $S_{L,c} \subset \pi_n(Y)$ to be
    $$\{\alpha \in \pi_n(Y): \lvert p_1(\alpha) \rvert \leq cL^{k+1}\text{ and }
    \lvert p_2(\alpha) \rvert \leq cL^k\}.$$
    Then there are constants $C>c>0$ such that the set of elements of $\pi_n(Y)$
    which have $L$-Lipschitz representatives lies between $S_{L,c}$ and $S_{L,C}$.
  \end{thm}

  To show part (i), we take two $L$-Lipschitz maps (which we can assume to be
  cellular) and let $k$ be the first dimension in which they differ.  Then the
  obstruction to homotoping one to the other is an element of $C^k(X;\pi_k(Y))$
  where every cell $\sigma$ is sent to a $C_\sigma L$-Lipschitz map $S^k \to Y$.
  Then Theorem \ref{thm:dist} gives the right kind of bound on the size of this
  obstruction.

  Now we show part (ii).  Given an obstruction-theoretic description of a
  homotopy class $\alpha \in [X,Y]$, we build a map $f:X \to Y$ by skeleta.
  First send $X^{(1)}$ to the basepoint.  Now suppose we have constructed
  $f_{k-1}:X^{(k-1)} \to Y$ with Lipschitz constant $M$.  We extend to $X^{(k)}$ by
  first using the bound of Theorem \ref{thm:sym} (or that of Conjecture
  \ref{conj:sym}) to fill in the $k$-simplices as efficiently as possible, then
  in each $k$-simplex $p$ glue in an efficient representative of the element
  $\alpha_p \in \pi_k(X)$ corresponding to the resulting obstruction.  The
  Lipschitz constant of the resulting map is
  $$C_k\max\{\gamma_k(M),\sup \{\Lip\alpha_p: k\text{-simplices $p$ of }X\}\},$$
  where $\gamma_k$ is our bound on sizes of homotopies.

  If all the $\gamma_k$ are linear, we get a linear bound.  Otherwise, we
  accumulate a Lipschitz constant
  $$\gamma_{\dim X}(\cdots(\gamma_3(\gamma_2(L)))\cdots),$$
  but this is still $O(L^{1+\epsi})$ for any $\epsi>0$.

  Now, maps in the same rational homotopy class may not be specified by the same
  rational obstructions.  However, they have the same optimal Lipschitz constant
  by Theorem 5.1 of \cite{PCDF}.
\end{proof}

\subsection*{Acknowledgments}
The author was partially supported by NSF RTG grant DMS-1547357.  He would like
to thank the anonymous referee for helpful suggestions regarding the exposition.

\bibliographystyle{amsalpha}
\bibliography{liphom}
\end{document}